\documentclass[11pt,oneside]{amsart}
\usepackage{bm, calc,geometry, verbatim, graphicx, amssymb, color, mathabx, mathtools, enumerate, bm, mathrsfs, amsmath, amsthm, ulem, xspace,tabularx}
\usepackage[usenames,dvipsnames,svgnames,table]{xcolor}
\usepackage[pdftex,bookmarks,colorlinks,breaklinks]{hyperref} 

\hypersetup{linkcolor=blue,citecolor=red,filecolor=dullmagenta,urlcolor=darkblue} 
\newtheorem{theorem}{Theorem}[section]
\newtheorem{corollary}[theorem]{Corollary}
\newtheorem{lemma}[theorem]{Lemma}
\newtheorem{conjecture}[theorem]{Conjecture}

\usepackage{physics}

\newcommand\dg{\operatorname{\textup{{\fontfamily{ptm}\selectfont deg}}}}

\newcommand{\cc}{\mathcal}

\newcommand{\cG}{\mathcal{G}}

\newcommand{\eps}{\varepsilon}    
\usepackage{bm}
\newcommand\rk{\operatorname{\textup{{\fontfamily{ptm}\selectfont rk}}}}
\newcommand\crk{\operatorname{\textup{{\fontfamily{ptm}\selectfont crk}}}}

     \makeatletter
      \def\@setcopyright{}
      \def\serieslogo@{}
      \makeatother      
\begin{document}
\author{Amin  Bahmanian}
   \address{Department of Mathematics,
  Illinois State University, Normal, IL USA 61790-4520}
 \author{Songling Shan}\thanks{Songling Shan was partially supported by NSF  grant 
 	DMS-2345869.}
   \address{Department of Mathematics and Statistics,
  Auburn University, Auburn, AL USA 36849}
   \email{szs0398@auburn.edu}

  \title[Spanning Euler Tours in Hypergraphs]{Spanning Euler Tours in Hypergraphs}
    
   \begin{abstract}  
%Can we find a cyclic ordering of blocks of a design so that every two consecutive blocks intersect? Can we further do this order this so that the intersections represent all the points? This has  applications in a number of fields, for example  in disk erasure codes \cite{MR1981142}. 
Motivated by generalizations of de Bruijn cycles to various combinatorial structures (Chung, Diaconis, and Graham), we study various Euler tours in set systems. Let $\cc G$ be a hypergraph whose corank and rank are $c\geq 3$ and $k$, respetively. The minimum $t$-degree of $\cc G$ is the fewest number of edges containing every $t$-subset of vertices. 
An {\it Euler tour} ({\it family}, respectively) in  $\cc G$ is a (family of, respectively) closed walk(s)  that (jointly, respectively) traverses each edge of $\cc G$ exactly once. An Euler tour is {\it spanning} if it traverses all the vertices of $\cc G$. We show that $\cc G$ has an Euler family if its incidence graph is $(1+\lceil k/c \rceil)$-edge-connected.  Provided that the number of vertices of $\cc G$ meets a reasonable lower bound, and either $2$-degree is at least $k$ or $t$-degree is at least one for $t\geq 3$, we show that $\cc G$ has a spanning Euler tour.

To exhibit the usefulness of our results, we solve a number of open problems concerning ordering blocks of a design (these have applications in other fields such as erasure-correcting codes). Answering a question of Horan and Hurlbert, we show that a Steiner quadruple system of order $n$ has a (spanning)  Euler tour if and only if $n\geq 8$ and $n\equiv 2,4 \pmod 6$, and we prove a similar result for all Steiner systems, as well as all designs except for 2-designs whose index $\lambda$ is less than the largest block size. We nearly solve a conjecture of Dewar and Stevens on the existence of universal cycles in pairwise balanced designs. Motivated by R.L. Graham's question on the existence of Hamiltonian cycles in block-intersection graphs of Steiner triple systems, we establish the Hamiltonicity of the block-intersection graph of a large family of (not necessarily uniform) designs.
All our results are constructive and of polynomial time complexity. 

\end{abstract}
   \subjclass[2010]{05C65, 05C45, 05C70, 05C40, 05C62, 05B05}
 \keywords{ spanning Euler tour;  eulerian hypergraph; Euler family; quasi-eulerian hypergraph;   discharging, universal cycle, $(g,f)$-factor}

 \date{\today}

   \maketitle    
  
\section{Introduction} %these  generalize  de Bruijn cycles
% Since Eulerian graphs are 2-edge-connected, 
Despite Euler tours being the subject of the first paper in graph theory \cite{euler1741}, and despite their numerous and wide range of applications from bioinformatics to reconstruct the DNA sequence \cite{PevznerDNA} to coding theory \cite{MR1491049}, the corresponding notion for hypergraphs is not well-understood. A   source of interest in Euler tours in hypergraphs is universal cycles for set systems; these are special Euler tours in complete uniform hypergraphs, and were introduced by Chung, Diaconis, and Graham \cite{MR1197444}. 
 Another  source of interest is applications in geographic information systems; in this context,  Euler tours  have been studied for 3-uniform hypergraphs \cite{BartGold1, MR2057782}.  
 Our goal is to  (i) characterize  uniform quasi-Eulerian hypergraphs of corank at least three (see Theorem \ref{quasimainthm}), and (ii) obtain  minimum degree conditions under which a (not necessarily uniform) hypergraph (possibly with multiple edges) has a spanning  Euler tour (see Theorem \ref{mainlong}). To elaborate, we need some background.

A {\it hypergraph} $\mathcal G$ is a pair $(V ,E)$ where $V$ is a finite set called the {\it vertex} set, $E$ is the {\it edge} (multi)set, where every edge is  a sub(multi)set of $V$. We write $\mu(\cc G)$ for the maximum of the edge multiplicities in $\cc G$, and $c(\cc G)$ denotes the number of components of $\cc G$.   The {\it rank} and {\it corank} of $\mathcal G$, written $\rk(\mathcal G)$ and $\crk(\mathcal G)$, respectively, are $\max \{|e|: e\in E\}$ and $\min \{|e|: e\in E\}$, respectively, and $\cG$ is {$k$}-uniform if $\rk(\mathcal G)=\crk(\mathcal G)=k$.  For  $T\subseteq V(\cG)$, the {\it degree} of $T$, written $\dg_\cG(T)$, is the number of edges containing $T$, abbreviate $\dg_\cG(\{v\})$ to $\dg_\cG(v)$ for $v\in V$, and {\it the minimum $k$-degree}, written $\delta_k(\cG)$, is $\min \{\dg_\cG(T): T\subseteq V(\cG), |T|=k\}$. The maximum $k$-degree, written $\Delta_k(\cG)$, is defined in a similar way, and the $k$-degree of $\cc G$ is $d$ if $\Delta_k(\cG)=\delta_k(\cG)=d$. 
A {\it flag} in a hypergraph $\cG$ is a pair $(v,e)$ where $v\in V, e\in E$ and $v\in e$. A connected hypergraph is $k$-flag-connected if it remains connected whenever fewer than $k$ flags are removed, 
where given a set  $F$ of $k$ flags of $\cG$,   removing flags in $F$ means that for all edges $e\in E(\cG)$,  if letting  $V_F(e) =\{v\in e: (v,e) \in F\}$, then we replace $e$ by $e\setminus V_F(e)$. 
Any $k$-edge-connected hypergraph is $k$-flag-connected, but the converse is not true in general.  For basic hypergraph terminology, we refer the reader to \cite{MR3578165}.

%(so $e_1,\dots,e_t$ are pairwise distinct, and $E=\{e_1,\dots,e_t\}$) 
Let $\mathcal G=(V,E)$ be a connected  hypergraph. Let the {\it closed walk} $W:=v_1,e_1,v_2,\dots,v_t,e_t$ be an alternating sequence  of  vertices $v_i$ and  edges $e_i$ of $\cG$ with 
 $\{v_i,v_{i+1}\}\subseteq e_i$ and $v_i\neq v_{i+1}$ for $i\in [t]:=\{1,\dots,t\}$ (indices are taken module $t$). We say that $W$ {\it traverses} a flag $(v,e)$ if for some $i\in [t]$ we have $v=v_i$ and  $e\in \{e_{i-1},e_i\}$. Similarly, $W$ traverses an edge $e$ (a vertex $v$, respectively) if for some $i\in [t]$, $e=e_i$ ($v=v_i$, respectively).   If $W$ traverses each flag of $\cG$ exactly once, then  we say that $W$ is a {\it flag spanning tour}. 
If $W$ traverses each edge of $\cG$ exactly once, then we say that $W$ is an {\it Euler tour}, and we say that $\cG$ is {\it Eulerian}. A {\it spanning Euler tour} is an Euler tour that traverses each vertex of $\cc G$. A family $\mathscr W:=\{W_1,\dots,W_k\}$ of closed walks in $\cG$ is an {\it Euler family} if each edge of $\cG$ is traversed by exactly one element of  $\mathscr W$, and   we say that $\cG$ is {\it quasi-Eulerian}. 
All  four notions coincide if $\cG$ is a graph. Moreover, every Eulerian hypergraph is quasi-Eulerian, but the converse is not true in general.

%(or spanning Euler tour) however, it is not yet easy to determine whether or not a given hypergraph is quasi-Eulerian. ---known as universal cycles--- 
A hypergraph has a  flag spanning tour if and only if all of its vertex degrees and all edge sizes are even \cite{MR3691547}. The problem of deciding whether a given hypergraph is quasi-Eulerian is polynomial \cite{MR3691547}, but the problem of deciding whether a given (linear 3-uniform) hypergraph is Eulerian is NP-complete \cite{MR2745697,MR3691547}, and so  hypergrapahs  with a (spanning) Euler tour are difficult to characterize. One may define Euler tours  based on the size of intersection of consecutive edges in a walk \cite{MR3863294}; the remarkable  result of \cite{MR4181762}  settles an important conjecture  on the existence of such tours  in complete uniform hypergraphs \cite{MR1197444}.  
%Are there any $\cc G$ with corank at least 3 and without  strong cut edges that is not quasi-Eulerian? Is 2-flag-connectivity necessary for a hypergraph to be quasi-Eulerian? 
%We need to find conditions under which a hypergraph without strong cut edges that is not 2-flag-connected is quasi-Eulerian. It remains an open question  which 2-edge-connected uniform hypergraphs are Eulerian. 
A natural approach to find an Euler tour is to  carefully modify  an Euler family to make it connected.  In fact, this was the initial motivation in \cite{MR3691547}  to introduce Euler families and classify quasi-Eulerian hypergraphs via $(g,f)$-factors (although, it is not yet easy to determine whether or not a given hypergraph is quasi-Eulerian). An edge $e$ is a {\it strong cut edge} in $\cc G$ if $c(\cc G-e)=c(\cc G)+|e|-1$.  It is known that  2-edge-connected 3-uniform hypergraphs, and  $k$-uniform hypergraphs ($k\geq 3$) of minimum 2-degree one  are quasi-Eulerian \cite{MR3691547, MR4622724}. Since  quasi-Eulerian hypergraphs have no strong cut edges \cite[Lemma 2]{MR3691547}, the following result leaves a small gap in classifying all uniform quasi-Eulerian hypergraphs of corank at least three. 
    \begin{theorem} \label{quasimainthm}
  	Any $\left(1+\left\lceil \rk(\cG)/\crk(\cG) \right\rceil\right)$-flag-connected hypergraph $\cc G$ with $\crk(\cG)\geq 3$ is quasi-Eulerian. 
  \end{theorem} 
\noindent In particular,  any 2-flag-connected (and consequently, any 2-edge-connected) $h$-uniform hypergraphs with $h\geq 3$ is quasi-Eulerian. Since minimum 2-degree one implies 2-edge-connectivity, this special case of Theorem \ref{quasimainthm} provides a short proof of the main result of \cite{MR4622724}. 
We  remark that it is immediate from definitions that if $\cc G$ is 2-flag-connected, then it has no  strong cut edges, but the converse is not true in general. 

%(and in particular, triple systems) %If $t=2$, then we call it a balanced incomplete block designs. A uniform PBD is a BIBD.
The {\it line graph} ({\it $\ell$-line graph}, respectively) $\mathscr L(\cc G)$ ($\mathscr L_\ell(\cc G)$, respectively) of a hypergraph $\cc G=(V,E)$ is a graph whose vertex set is $E$, and distinct $e,f\in E$ are adjacent in  $\mathscr L(\cc G)$ if $e\cap f\neq \emptyset$ ($|e\cap f|=\ell$, respectively) in $\cc G$. Steimle and \v Sajna \cite{MR3843267} studied spanning Euler tours in hypergraphs with particular vertex cuts \cite{MR3843267}. Lonc and Naroski characterized $k$-uniform hypergraphs with a connected  $(k-1)$-line graph that are Eulerian \cite{MR2745697}. By gluing together the Euler families of  2-edge-connected 3-uniform hypergraphs in \cite{MR3691547},  Wagner and \v Sajna \cite{MR3619996}  were able to show  that triple systems are Eulerian, which improved the results of \cite{MR3217514}. This idea was further used to establish  that  every $k$-uniform hypergraph $H$ with $\delta_{k-1}(H)\geq 1$ is Eulerian \cite{MR4472768}.

A majority of the work on Eulerian hypergraphs is disguised under different terminology in design theory. A {\it $t$-$(v,k,\lambda)$ design} (or a {\it $t$-design}) is a $v$-vertex $k$-uniform hypergraph  whose $t$-degree is exactly $\lambda$. {\it Points},  {\it blocks}, {\it block-intersection graphs}, and {\it universal cycles of rank two} (or   {\it 1-overlap cycles}) are vertices,  edges, line graphs, and Euler tours, respectively. A {\it Steiner system} $S(t,k,v)$ is a   design with $\lambda=1$, and {\it Steiner triple systems}, and {\it Steiner quadruple systems} are Steiner systems with  $(t,k)=(2,3)$, and $(t,k)=(3,4)$, respectively. 
A {\it BIBD} is a 2-design, and a {\it $(v,K,\lambda)$-PBD} is a $v$-vertex hypergraph whose $2$-degree is $\lambda$, and that whose edge cardinalities are in a set of integers $K$.  Ordering blocks of a design  has striking applications; for example ordering blocks of a Steiner triple system  leads to erasure-correcting codes \cite{MR1981142}. Since the line graph of an Eulerian hypergraph is Hamiltonian, another motivation to  order blocks of a design is R. L. Graham's question on the Hamiltonicity of the block-intersection graph of  Steiner triple systems \cite{MR1079733,MR0982111,MR2831806,MR3514754,MR2551985}. A notable conjecture is the following (see \cite[Conjecture 6.2.1]{MR2711021}, or \cite[Conjecture 5.6]{MR2963059}).
\begin{conjecture} \cite{MR2963059}  \label{conjdewar}%:  
    Every $(v,K,\lambda)$-PBD with $\min K\geq 3$ and a sufficient number of non-regular base blocks admits a universal cycle of rank two. 
\end{conjecture}

An $n$-vertex hypergraph $\cc G$ is {\it $(n,c,k,\mu)$-admissible} if  $\crk(\cc G)=c\geq 3, \rk(\cc G)=k, \mu(\cc G)=\mu$, and  $n\geq g(c,k,\mu)$ where $g(3,3,\mu)=7$, $g(3,4,\mu)=10$,  $g(c,k,\mu)=\binom{k}{2}+1$ if $c\geq 4, k\leq 2c-2$,  $g(c,k,\mu)=(4c^2+k^2-3k+2)/(4c-2k+2)$ if $c\geq 3,2c-1\leq k\leq 2c$, and $g(c,k,\mu)=2\mu(k-1)(2^{k-c-1}-2^c+1)+\binom{k}{2}+1$ if $k>2c$.
Here is our next result.
\begin{theorem} \label{mainlong}
An $(n,c,k,\mu)$-admissible hypergraph $\cc G$ 
  has a  spanning  Euler tour if any of the following  conditions hold.
 \begin{itemize}
     \item [(i)] $\delta_2(\cc G)\geq k$;
     \item [(ii)] $\delta_3(\cc G)\geq 1$, $n\geq k^2-3k+5$;
     \item [(iii)] $\delta_r(\cc G)\geq 1$, $4\leq r\leq k$.     
 \end{itemize} 
\end{theorem}
%Spanning is more applicable and relatable to universal cycles. in particular, We are not aware of any results on Hamiltonicity of designs with $t\geq 3$)
Theorem \ref{mainlong} complements existing results on minimum degree conditions that ensure a uniform hypergraph contains a   Hamiltonian cycle \cite{MR3150175, MR1671170, MR2274077, MR2195584}, or a perfect matching \cite{MR2260124, MR2500161}. We also remark that Theorem \ref{mainlong} is new even if we restrict ourselves to (not necessarily spanning) Euler tours  in uniform hypergraphs. 

\begin{corollary}\label{maincor1}
An $n$-vertex $k$-uniform hypergraph $\cc G$ 
  has a  (spanning)  Euler tour if any of the following  conditions hold.
 \begin{itemize}
     \item [(i)] $k=3, n\geq 7, \delta_2(\cc G)\geq 3$;
     \item [(ii)] $k\geq 4$ and either $\delta_2(\cc G)\geq k$, $n\geq \binom{k}{2}+1$ or $\delta_3(\cc G)\geq 1$, $n\geq k^2-3k+5$;
     \item [(iii)] $k\geq 5, n\geq \binom{k}{2}+1, \delta_r(\cc G)\geq 1$, $4\leq r\leq k$.     
 \end{itemize} 
\end{corollary}

Theorem \ref{mainlong} settles the problem of finding universal cycles of rank two for the vast majority of (not necessarily uniform) designs; The only gap is for 2-designs whose index $\lambda$ is less than the largest block size. %Here, we provide a  few special families of designs  for which finding the universal cycles of rank two was partially known. 
Previously it was unknown if all Steiner quadruple systems are Eulerian  \cite{MR3145634}. Up to isomorphism, there is only one Steiner quadruple system of order 8 \cite{Barrau} (which has a spanning Euler tour), and the Steiner quadruple systems of order 2 and 4 are not Eulerian, hence we have the following by noting that $n\equiv 2, 4 \pmod {6}$ is a necessary condition for the existence of Steiner quadruple systems. 
%2148, 8523, 3684, 4578, 8156, 6287, 7813, 3576, 6471, 1572, 2163, 3274, 4135, 5462.
\begin{corollary}
  A Steiner quadruple system of order $n$ has a spanning  Euler tour if and only if $n\geq 8$ and $n\equiv 2,4 \pmod 6$.  
\end{corollary}
More generally, we have the following. 
\begin{corollary}
  Any  Steiner system $S(t,k,v)$ has a spanning  Euler tour for  $t=3, k\geq 4, n\geq k^2-3k+5$, and $t\geq 4, k\geq 5, n\geq \binom{k}{2}+1$.
\end{corollary}

Regarding Conjecture \ref{conjdewar}, Theorem \ref{mainlong} implies the following. 
\begin{corollary} \label{dewstevconjcor}
    Every $(v,K,\lambda)$-PBD with $\min K\geq 3$, $\lambda\geq \max K$ and a sufficient number of points admits a (spanning) universal cycle of rank two. 
\end{corollary}
\noindent Dewar and Stevens  made a similar conjecture for BIBDs \cite[Conjecture 5.8]{MR2963059}, but since every BIBD is a PBD, Corollary \ref{dewstevconjcor} addresses the second conjecture as well.

Last but not least, for the vast majority of designs, it was  unknown if the corresponding block-intersection graph is Hamiltonian. Most of the literature is focused on 2-designs \cite{MR1140783,MR1079733,MR2378924,MR1312454,MR3754043}, and in fact we are aware of only one result for $t$-designs with $t\geq 3$ (The Hamiltonicity of the block intersection graph of 
$t$-$(v,t+1,\lambda)$ designs was proven in \cite{MR2831806}). Theorem \ref{mainlong} establishes the Hamiltonicity of  the block-intersection graph of a large family of (not necessarily uniform) designs (Again, it leaves the small gap where $t=2$ and $\lambda < \max K$). Since deciding if a graph is Hamiltonian is NP-complete, it is worth mentioning that Abueida and Pike showed that finding Hamiltonian cycles in   block intersection graph of  a BIBD can be done in polynomial time \cite{MR3055152}. All our results on Euler families and spanning Euler tours (and consequently, on Hamiltonicity of the  line graph of the hypergraph under consideration) are constructive and are of 
polynomial time complexity.

In Section~\ref{term},  we provide the  prerequisites. In Sections \ref{quasisec}, and \ref{spaneulersecr2}, 
we prove Theorems \ref{quasimainthm}, and \ref{mainlong}, respectively. Our proofs rely on discharging method and factor theory. The lower bounds on $n$ are chosen so as to simplify the calculations, and could be improved by a more careful analysis. 

\section{Preliminaries} \label{term}
Let $r\in \mathbb{N}$, and let $G=(V,E)$ be a graph with $x\in V, A,B\subseteq V$. Then $N_G(A)$ denotes the set of vertices having a neighbor in $A$, abbreviate $N_G(\{x\})$ to $N_G(x)$. The graph $G$ is {\it even} if all its vertices have even degree, and  is $(A,r)$-regular if $\dg_G(x)=r$ for each $x\in A$.  The graph $G-A$ is obtained by removing $A$ and all the edges incident with $A$ from $G$. Finally,  if $A\cap B=\emptyset$, $\eps_G(A,B)$ denotes the number of edges with one endpoint in $A$ and the other endpoint in $B$, abbreviate $\eps_G(A, \{x\})$ to $\eps_G(A,x)$, and $\eps_G(A,W)$ to $\eps_G(A,H)$ where $G\supseteq H:=(W,F)$.  
A bipartite graph $G$ with bipartition $\{X, Y \}$ will be denoted by $G[X, Y ]$, and for $A\subseteq X,B\subseteq Y$, denote $\overline A=X\backslash A, \overline B=Y\backslash A$.   The {\it incidence graph} $\mathscr I(\cG)$ of a hypergraph  $\cG=(V,E)$ is a bipartite  graph with bipartition $\{ V, E \}$ and edge set $\{ ve: v\in V, e \in E, v \in e\}$. Although $\cc G$ may have multiple edges, its incidence graph is always simple.  
An Euler family in $\cG$ corresponds to an even $(E,2)$-regular subgraph of $\mathscr I(\cG)$, and an Euler tour corresponds to a connected even $(E,2)$-regular subgraph of $\mathscr I(\cG)$. 

\subsection{Parity $(g,f)$-factors and Barriers} \label{prereq}
Let $G=(V,E)$ be a graph. For $U\subseteq W\subseteq V$, a component $C$ of $G- W$ is  {\it $U$-odd} if $\eps_G(U,C)$ is odd, and  a component that is not $U$-odd   is {\it 
$U$-even}.    
For $G[X,Y]$, and  $S\subseteq X$, $T\subseteq X\cup Y$ with $S\cap T=\varnothing$, let 
$$
\delta(S,T)=2|S|+\sum\nolimits_{v\in T}\dg_{G-S}(v)-2|T\cap X|-q(S,T),
$$
where $q(S,T)$ is the number of $T$-odd components of $G- (S\cup T)$. If $\delta(S,T)<0$, then  $(S,T)$  is a {\it barrier} of {\it size} $|S\cup T|$ in $G$, and a {\it minimum barrier} is a barrier whose size is as small as possible. Observe that if $(S,T)$  is a barrier, then $T\neq \varnothing$.  We say that a vertex $v\in T$ is {\it adjacent} in $G$ to a component of $G-(S\cup T)$ if $v$ is adjacent in $G$ to a vertex from the component. 
\begin{lemma} \label{easyprop}
For  $G[X,Y]$, and    $S\subseteq X$, $T\subseteq X\cup Y$ with $S\cap T=\varnothing$, $\delta(S,T)$ is  always even. 
\end{lemma}
\begin{proof} Let $f$ be the number of edges in $G-S$ whose both endpoints are in $T$. 
Let $C_1,\dots,C_k$ be the components of $G-(S\cup T)$. For $i\in [k]$, define $\delta_{C_i}=1$ if $C_i$
is $T$-odd, and  $\delta_{C_i}=0$ if $C_i$ is $T$-even. Since $\eps_G(T,C_i)-\delta_{C_i}\equiv 0 \pmod 2$ for  $i\in [k]$,  we have 
\begin{align*}
\delta(S,T)&\equiv\sum\nolimits_{v\in T}\dg_{G-S}(v)-q(S,T)\\
& =2f+\sum\nolimits_{i\in [k]} \eps_G(T,C_i)-q(S,T)\\
&\equiv \sum\nolimits_{i\in [k]} (\eps_G(T,C_i)-\delta_{C_i}) \equiv 0 \pmod 2.
\end{align*}
\end{proof}

Using  Lovasz's parity $(g,f)$-factor Theorem \cite{MR325464}, Bahmanian and \v Sajna proved the following result. 
\begin{theorem}\textup{(\cite[Corollary 6.2]{MR3691547})}\label{parity-factor}
Let $\cG=(V,E)$ be a hypergraph and $G[X,Y]=\mathscr I(\cG)$ with $X:=E,Y:=V$. Then $\cG$ is quasi-eulerian if and only if  all  $S \subseteq X$ and $T\subseteq X \cup Y$ with $S\cap T=\varnothing$ satisfy $\delta(S,T)\ge 0$. 
\end{theorem}
We need the following very useful lemma. 
\begin{lemma}\label{minimum-barrier}
If $(S,T)$ is a minimum barrier of $G[X,Y]$, then 
\begin{enumerate}[(i)]
\item [\textup{(i)}] $T\subseteq X$;% (and so $T$ is an independent set in $G$);
\item [\textup{(ii)}] For each $T$-even component   $C$  of $G-(S\cup T)$,   $\eps_G(T,C)=0$;  
\item [\textup{(iii)}] For each   $T$-odd component $C$ of $G-(S\cup T)$ and $u\in T$, $\eps_G(u,C)\leq 1$;
\item [\textup{(iv)}] If $O_1,\dots,O_{q}$ are the $T$-odd components of $G- (S\cup T)$ where $q:=q(S,T)$, then
\begin{equation*}%\label{ineq1}
|T|-|S|>\frac{1}{2} \sum\nolimits_{i\in [q]}\Big(\eps(T,O_i)-1\Big).
\end{equation*}
\end{enumerate}
 \end{lemma}
\begin{proof}
Let us fix $u\in T$, and let $T^*=T\backslash\{u\}$. It is clear that
\begin{enumerate}  
\item [(a)]
$\sum\nolimits_{v\in T^*}\dg_{G-S}(v)=\sum\nolimits_{v\in T}\dg_{G-S}(v)-\dg_{G-S}(u)$, and 
\item  [(b)]
$|T^*\cap X|=|T\cap X|-1$ if $u\in X$ and $|T^*\cap X|=|T\cap X|$ if $u\in Y$.  
\end{enumerate}
Moreover, if  $\kappa$ is the number of odd components of $G-(S\cup T)$ that are adjacent to $u$ in $G$, then 
\begin{enumerate}  
\item [(c)] $\dg_{G-S}(u)\ge \kappa$, and 
\item  [(d)] $q(S,T)\in \{q(S,T^*)+\kappa, q(S,T^*)+\kappa-1 \}$.  
\end{enumerate}
Since $(S,T)$ is a minimum barrier, by (a), (b), and (d) we have  
\begin{align*}
0& \leq \delta(S,T^*)=2|S|+\sum\nolimits_{v\in T^*}\dg_{G-S}(v)-2|T^*\cap X|-q(S,T^*)\\ 
&\leq  \delta(S,T)-\dg_{G-S}(u)+\kappa+
\begin{cases}
2 & \text{if}\,\, u\in  X, \\
0  & \text{if}\,\, u\in Y. 
\end{cases}
\end{align*}
Since $(S,T)$ is a barrier, by Lemma \ref{easyprop} we have  $\delta(S,T)\leq -2$. Thus, 
\begin{align*}
\dg_{G-S}(u)\leq \begin{cases}
 \kappa & \text{if}\,\, u\in X, \\
\kappa-2 & \text{if}\,\, u\in Y. 
\end{cases}
\end{align*}
 Combining this with (c) we have  $\kappa \leq \dg_{G-S}(u)\leq \kappa-2$ for $u\in Y$. Therefore,  $T\cap Y=\varnothing$, or equivalently, $T\subseteq X$. Since $X$ is an independent set, so is $T$. This proves (i). Moreover, for $u\in X$ we have  $\dg_{G-S}(u)= \kappa$, which proves (ii) and (iii). By  (i) and (ii), we have  
\begin{equation*}\label{size}
0>\delta(S,T)=2|S|-2|T|+\sum\nolimits_{i\in [q]}\eps(T,O_i)-q,
\end{equation*}
which proves (iv). 
\end{proof}
\subsection{Spanning Trees} \label{twolemmas}
In order to prove Theorem~\ref{mainlong}, we need the following technical lemma. For the purpose of this subsection, a spanning tree $T$ in a bipartite graph $G[X,Y]$ is {\it nice} if $\dg_T(x)\le 2$ for every $x\in X$. Let 
$$
\mu:= \max_{u\in X}|\{v\in X \ |\ N_G(u)=N_G(v)\}|.
$$
\begin{lemma}\label{LEM:Xdegree}
	Let $G[X,Y]$ be a simple bipartite graph. Suppose that $3\leq c\leq \dg_G(x)\leq k$ for each $x\in X$, and $|N_G(y)\cap N_G(z)|\geq k$ for every pair of distinct vertices $y,z\in Y$. We have the following.
 \begin{enumerate}
    \item [\textup{(i)}] $|N_G(y)|\geq \dfrac{ k (|Y|-1)}{k-1}$ for $y\in Y$;   
    \item [\textup{(ii)}] $|N_G(B)\cap N_G(\overline B)| \geq \dfrac{2|B| (|Y|-|B|)}{ k-1}$ for $B\subseteq Y$;  
    \item [\textup{(iii)}] $G$ has a nice spanning tree if $|Y| \ge \mu(k-1)(2^{k-c-1}-2^c+1)$.     
\end{enumerate}
\end{lemma} 
\begin{proof} To prove (i), we count  in two  ways $\eps_G(N_G(y), Y\backslash\{y\})$ for $y\in Y$. Each $x\in N_G(y)$ is adjacent to at most $k-1$ vertices in $Y\backslash\{y\}$, and each $z\in Y\backslash \{y\}$ is adjacent to at least $k$ vertices in $N_G(y)$. Thus, 
$$
k(|Y|-1)\leq \eps_G(N_G(y), Y\backslash\{y\})\leq  (k-1)|N_G(y)|.
$$

To prove (ii), let $B\subseteq Y$, and let
	$$
	S=\{(x,\{y,z\})\ |\ y\in B, z\in \overline B, x\in N_G(y)\cap N_G(z)\}.
	$$
	On the one hand, for fixed $y\in B, z\in \overline B$, there are at least $k$ choices for $x\in N_G(y)\cap N_G(z)$, and so $|S|\geq k|B| |\overline B|$.  On the other hand,  each $x\in N_G(B)\cap N_G(\overline B)$ has at most $k$ neighbors, and so there are at most $\binom{k}{2}$ choices for the set $\{y,z\}$ with $y\in B, z\in \overline B$, and consequently, $|S| \leq \binom{k}{2} |N_G(B)\cap N_G(\overline B)|$. Thus, 
	$$
	k|B| |\overline B| \leq \binom{k}{2} |N_G(B)\cap N_G(\overline B)|.
	$$

To prove (iii), let $\mathscr T=\{T\subseteq G\ |\ T \text{ is a tree, } \dg_T(x)\leq 2\  \forall x\in V(T)\cap X\}$. Since every edge of $G$ induces a tree that belongs to $\mathscr T$, we have  $\mathscr T\neq \varnothing$, so we may choose a tree  $T\in \mathscr T$ whose order is maximum. Let 
	\begin{align*} 
		&&
		A&=V(T)\cap X,
		&&
		B=V(T)\cap Y,\\
		&&
		A_1&=\{x\in A\ |\   \dg_T(x)=1\},
		&&
		A_2=A\backslash A_1.
		&&
	\end{align*}
	It is clear that $A\neq \varnothing, B\neq \varnothing$. In fact, since each vertex $x\in X$ and two of its neighbors induce a tree that belongs to $\mathscr T$, $|B|\geq 2$.   If $T$ is spanning, then there is nothing to prove. So we assume that  $T$ is not spanning, or equivalently, $\overline A\cup \overline B\neq \varnothing$. We claim that
	\begin{enumerate} [(a)]
		\item $|A_2|=|B|-1$;
		\item $\eps_G(A_1, \overline B)=\eps_G(\overline A, B)=0$;
		\item $\overline B\neq \varnothing$;
		\item $\eps_G(A_2, \overline B)\neq 0$;
		\item $|\overline A|\geq 1+ \dfrac{  |Y|-1}{k-1}$;
		\item $2|B| |\overline B|\leq (k-1) (|B|-1)$.
	\end{enumerate}
	Since $T$ is a tree, we have $$|A_1|+|A_2|+|B|=|V(T)|=|E(T)|+1=|A_1|+2|A_2|+1,$$ which proves (a). The maximality of $T$ implies (b).  	To prove (c), suppose on the contrary that $\overline B= \varnothing$. Then  $\overline A\neq \varnothing$. This together with $\eps_G(\overline A, B)=0$ imply that all vertices in $\overline A$ are isolated,
	which contradicts our hypothesis that $\dg_G(x)\geq 2$ for  $x\in \overline A$.  
		Combining our hypothesis and (b) implies that 
	$$\varnothing \neq N_G(B)\cap N_G(\overline B)\subseteq A_2,$$
	which proves (d). By (a)--(c) and Lemma~\ref{LEM:Xdegree}(i), we have the following which proves (e).
	\begin{align*}% \label{useiineq}
		|\overline A|+|Y|-2\geq |\overline A|+|B|-1&=|\overline A|+|A_2|  \geq  |N_G(\overline B)|\geq  \frac{k(|Y|-1)}{k-1}.
	\end{align*} 
Applying Lemma~\ref{LEM:Xdegree}(ii), we have the following which proves (f).
	\begin{align*} %\label{useiineq2}
	    	2|B| |\overline B| \leq (k-1) |N_G(B)\cap N_G(\overline B)| \leq (k-1)|A_2|=(k-1) (|B|-1).
	\end{align*}
	We obtain a contradiction to (f) if $|\overline B|  \geq  \frac{k-1}{2}$. Thus we assume $|\overline B|  < \frac{k-1}{2}$ and so $k> 2c+1$ by $|\overline B| \ge c$. 
%Since $A\neq \varnothing$ and $N_G(\overline A) \subseteq \overline B$,  if $k\leq 2c+1$, then  $|\overline B| \ge c\geq  \frac{k-1}{2}$ which contradicts (f), and consequently, $T$ is spanning. To complete the proof suppose $k> 2c+1$.
 Since $|Y| \ge \mu(k-1)(2^{k-c-1}-2^c+1)$,  then by (e)  we have $|\overline A| \ge \frac{|Y|-1}{k-1}+1 \ge \mu(2^{k-c-1}-2^c+1)$.
%  If on the contrary  $|\overline B| < \frac{k-1}{2}$, 
Now  we have the following which is a contradiction. 
\begin{align*} 
    |\overline A|&\leq \mu \sum\nolimits_{\ell=c}^{|\overline B|}\binom{|\overline B|}{\ell}=\mu\left(\sum\nolimits_{\ell=0}^{|\overline B|}\binom{|\overline B|}{\ell}-\sum\nolimits_{\ell=0}^{c-1}\binom{|\overline B|}{\ell}\right) \nonumber\\
    &\leq  \mu\left(2^{|\overline B|}-\sum\nolimits_{\ell=0}^{c-1}\binom{c}{\ell} \right)= \mu\left(2^{|\overline B|}-\sum\nolimits_{\ell=0}^{c}\binom{c}{\ell}+1 \right) \nonumber\\
    &< \mu\left(2^{(k-1)/2}-2^{c}+1 \right). 
\end{align*}
\end{proof}
We need the following generalization of Lemma \ref{LEM:Xdegree}(i). 
\begin{theorem} \label{enumlem}
For any $n$-vertex hypergraph $\cG$, and $0\leq i\leq j\leq \rk(\cG)=:k$, we have
$$
\dfrac{\delta_i(\cG)}{\delta_j(\cG)}\geq \binom{n-i}{j-i}\Big /\dbinom{k-i}{j-i}.
$$
\end{theorem}
\begin{proof}
Let $V$ and $E$ be the vertices and edges of $\cc G$. Let $I\subseteq V$ with $|I|=i$, and let $N_\cc G(I)$ denote the number of edges in $\cc G$ which contain $I$. We count in two  ways the size of the following set.
$$
S=\{(A,B)\ | \ A\subseteq V, |A|=j-i, A\cap I=\varnothing, B\in E, A\cup I\subseteq B  \}.
$$
On the one hand, there are $\binom{n-i}{j-i}$ choices for $A\subseteq V$ with $|A|=j-i, A\cap I=\varnothing$, and for each such $A$, there are at least $\delta_j(\cc G)$ choices for $B\in E$ with $A\cup I\subseteq B$. On  the other hand, there are $|N_\cc G(I)|$ choices for $B\in E$ with $I\subseteq B$, and for each such $B$, there are at most $\binom{k-i}{j-i}$ choices for $A\subseteq V$ with $|A|=j-i, A\cap I=\varnothing$,  $A\cup I\subseteq B$. Thus,
$$
\delta_j(\cc G)\binom{n-i}{j-i}\leq |S|\leq |N_\cc G(I)| \binom{k-i}{j-i}.
$$
To complete the proof, it suffices to take the minimum of the above over all $I$ with $|I|=i$. 
\end{proof}

Finally, we need the following two lemmas.
\begin{lemma} \label{ineqalg1}
   We have the following.
\begin{align*}
    &\max\left \{\binom{k}{2}+1, \frac{4c^2+k^2-3k+2}{4c-2k+2}\right\}\\
    &\quad\quad\quad=
    \begin{cases}
        1+\dbinom{k}{2}  &\mbox {if } 4\leq c\leq k\leq 2c-2,\\
        \dfrac{4c^2+k^2-3k+2}{4c-2k+2}  &\mbox {if } c\geq 3,2c-1\leq k\leq 2c.
    \end{cases}
\end{align*}
\end{lemma}
\begin{proof}
If $k\geq 5$, we have $k^2 - k -4 c + 2\geq k^2 - k -4 k + 2=k^2-5k+2\geq 0$, and so
\begin{align*}
  & 4 (4c-2k+2)\left ( \binom{k}{2}+1  \right) - 4(4c^2+k^2-3k+2)\\
&=(k^4 - 6 k^3 + 9 k^2 - 4 k + 4) -  ( k^2 - k -4 c + 2)^2\\
&\begin{cases}
    \geq (k^4 - 6 k^3 + 9 k^2 - 4 k + 4) -( k^2 - k -2k-4 + 2)^2=4k(k-4)\geq 0  &\mbox {if } k\leq 2c-2,\\
    \leq (k^4 - 6 k^3 + 9 k^2 - 4 k + 4) -( k^2 - k -2k-2 + 2)^2=4(1-k)\leq 0  &\mbox {if } k\geq 2c-1.\\
\end{cases}
\end{align*}
If $c=k=4$, then $\binom{k}{2}+1=\frac{4c^2+k^2-3k+2}{4c-2k+2}=7$. 
%If $c=3,k=5$, we have $\binom{k}{2}+1=11$ and $\frac{4c^2+k^2-3k+2}{4c-2k+2}=12$. 
%If $c=3,k=6$, we have $\binom{k}{2}+1=16$ and $\frac{4c^2+k^2-3k+2}{4c-2k+2}=28$. 
\end{proof} 
\begin{lemma} \label{ineqalg2}
If $\mu\geq 1,c\geq 3$ and $k\geq 2c+1$, then 
    $$2\mu(k-1)(2^{k-c-1}-2^c+1)+\dbinom{k}{2}+1\geq \frac{4 c^2 +5 k^2- 8 c k   - 3 k + 2}{2k-4 c  + 2}.$$ 
\end{lemma}
\begin{proof}
Since $\mu\geq 1$ and $k\geq 2c+1$, we have $2\mu(k-1)(2^{k-c-1}-2^c+1)+\binom{k}{2}+1\geq 2(k-1)+\binom{k}{2}+1$. 
    Since $k\geq 7$, we have $k^2 - k +4 c - 2\geq k^2 - k +12 - 2=k^2 - k +10\geq 0$, and so
\begin{align*}
  & 4 (2k-4 c  + 2)\left ( 2(k-1)+\binom{k}{2}+1  \right) - 4(4 c^2 +5 k^2- 8 c k   - 3 k + 2)\\
&= (k^4 + 2 k^3 - 7 k^2 + 20 k - 12) -  (k^2 - k +4 c - 2)^2\\
&    \geq (k^4 + 2 k^3 - 7 k^2 + 20 k - 12) -( k^2 - k +2k-2 - 2)^2\\
&=28(k-1)\geq 0.
\end{align*}
\end{proof}

%Proof of Theorem~\ref{quasimainthm}}
\section{Quasi-Eulerian Hypergraphs} \label{quasisec}
In this section, we prove Theorem \ref{quasimainthm}. Let  $\cc G$ be a $\left(1+\left\lceil k/c \right\rceil\right)$-flag-connected hypergraph where $c:=\crk(\cG)\geq 3$ and  $k:=\rk(\cG)$. We  show that $\cc G$ is quasi-Eulerian.  Let $G[X,Y]=\mathscr I(\cG)$ with $X=E, Y=V$. 
It is clear that $3\leq c\leq \dg_G(x)\leq k$ for  $x\in X$, and that $G$ is $(1+\lceil k/c \rceil)$-edge-connected. 
It suffices to show that $G$ has an even $(X,2)$-regular subgraph. Suppose on the contrary that $G$ does not have an even $(X,2)$-regular subgraph.   
By Theorem \ref{parity-factor}, $G$ has a barrier. Let $(S,T)$ be a minimum barrier of $G$ (so $S,T\subseteq X$ and $S\cap T=\varnothing$) and let $O_1,\dots,O_{q}$ be the $T$-odd components of $G- (S\cup T)$ where $q:=q(S,T)$. 

We define the auxiliary bipartite graph $H[S\cup T, \{o_1,\dots,o_q\}]$ %whose vertex set is $S\cup T\cup\{o_1,\dots,o_q\}$ 
such that 
\begin{align*}
\eps_{H}(x,o_i)= \eps_G(x,O_i), \mbox{ for }x\in S\cup T, i\in [q].
\end{align*}
In what follows, we shall define $f_0,f_1,f_2:V(H)\rightarrow \mathbb{R}$  such that $$\sum\nolimits_{x\in V(H)}f_0(x)=\sum\nolimits_{x\in V(H)}f_1(x)=\sum\nolimits_{x\in V(H)}f_2(x).$$ 
We let 
\begin{align*}
f_{0}(x)= \begin{cases}
c & \text{if}\,\, x\in S, \\
0 & \text{if}\,\, x\in T, \\
\dfrac{c}{2}\big(\eps_H(T,o_i)-1\big) & \text{if}\,\, x=o_i, i\in [q],
\end{cases}
\end{align*}
and
\begin{align*}
f_{1}(x)= \begin{cases}
0 & \text{if}\,\, x\in S\cup T, \\
f_0(x)+ \sum\nolimits_{u\in S} \dfrac{\eps_H(u,x)}{\dg_H(u)}f_0(u)         & \text{if}\,\, x=o_i, i\in [q].
\end{cases}
\end{align*}
Intuitively speaking, we transfer the value of $f_0(x)$ for $x\in S$ to $o_i$s so that 
 every vertex $x\in S$ gives $\frac{\eps_H(x,o_i)}{\dg_H(x)}f_0(x)$ to each $o_i$ for $i\in [q]$. Since $\dg_H(x)=\sum\nolimits_{i=1}^q \eps_H(x,o_i)$, we have $f_{1}(x)=0$ for $x\in S$.  Observe that
$$f_1(o_i)\geq f_0(o_i)=\dfrac{c}{2}\big(\eps_H(T,o_i)-1\big) \mbox { for } i\in [q].$$

Let $I=\{i\in [q]\ |\  \eps_H(T,o_i)=1\}, J=[q]\backslash I$. Since $G$ is $1+\left \lceil k/c \right\rceil$-edge-connected, and for $i\in I$, $\eps_G(T,O_i)=1$, we have 
$$\sum\nolimits_{x\in S} \eps_H(x,o_i)=\eps_G(S,O_i)\geq  \left \lceil \frac{k}{c} \right\rceil \mbox { for } i\in I.$$
Therefore,  for $i\in I$
\begin{align*}
f_1(o_i)&= \sum\nolimits_{x\in S} \dfrac{c\eps_H(x,o_i)}{\dg_H(x)}\geq \frac{c}{k} \sum\nolimits_{x\in S} \eps_H(x,o_i)
\geq  \frac{c}{k}   \left\lceil \frac{k}{c}\right \rceil\geq 1,
\end{align*}
where the first inequality is obtained by the fact that $\dg_H(x) \leq \dg_G(x)\leq k$ for  $x\in S$.

In order to define $f_2$, we transfer the value of $f_1(o_i)s$ to $T$ so that  
every vertex $o_i$ for $i\in [q]$ gives $\frac{\eps_H(x,o_i)}{\eps_H(T,o_i)}f_1(o_i)$ to each $x\in T$. More precisely,
\begin{align*}
f_{2}(x):= \begin{cases}
0 & \text{if}\,\, x\in S\cup \{o_1,\dots,o_q\}, \\
\sum\nolimits_{i\in [q]} \dfrac{\eps_H(x,o_i)}{\eps_H(T,o_i)}f_1(o_i)         & \text{if}\,\,  x\in T.
\end{cases}
\end{align*}

Fix $x\in T$, and let $I_x=\{i\in I\ |\  \eps_H(x,o_i)=1\}, J_x=\{i\in J\ |\ \eps_H(x,o_i)=1\}$. Observe that $\dg_G(x)=\dg_H(x)=|I_x|+|J_x|$, and that for $i\in J$, $\eps_H(T,o_i)\geq 3$ and so 
$\frac{\eps_H(T,o_i)-1}{\eps_H(T,o_i)}\geq 2/3$. 
We have 
\begin{align*}
f_2(x)&= \sum\nolimits_{i\in [q]} \frac{\eps_H(x,o_i)}{\eps_H(T,o_i)}f_1(o_i)\\
&=\sum\limits_{i\in I_x} f_1(o_i)+\sum\nolimits_{i\in J_x} \frac{f_1(o_i)}{\eps_H(T,o_i)}\\
&\geq  |I_x|  + \frac{c}{2}\sum\nolimits_{i\in J_x}\frac{\eps_H(T,o_i)-1}{\eps_H(T,o_i)} \\
&\geq  |I_x|  + \frac{c}{2}\big(\dg_G(x)-|I_x|\big) \frac{2}{3}\\
&\geq  \dg_G(x)\geq c.
\end{align*}
Therefore,
\begin{align*}
c|T|&\leq  \sum\nolimits_{x\in T}f_2(x)= \sum\nolimits_{x\in S}f_0(x)+\sum\nolimits_{i\in [q]} f_0(o_i)\\
&= c|S|+ \frac{c}{2}\sum\nolimits_{i\in [q]}\big(\eps_H(T,o_i)-1\big).
\end{align*}
Thus,  
$$|T|-|S|\leq \frac{1}{2} \sum\nolimits_{i\in [q]}\big(\eps(T,O_i)-1\big),$$
contradicting  Lemma~\ref{minimum-barrier}(iv). Therefore, $G$ has an even $(X,2)$-regular subgraph, and consequently, $\cG$ is quasi-Eulerian.
\qed

%Proof of Theorem~\ref{mainlong} for $r=2$
\section{Spanning Euler Tours} \label{spaneulersecr2}
In this section, we prove Theorem \ref{mainlong}. In subsections \ref{consspaneu}--\ref{dischspaneu}, and \ref{spaneufinal}, we settle the cases  $r=2$, and $r>2$, respectively.
\subsection{Construction of  Spanning Euler Tours} \label{consspaneu}
Let $\cc G$ be an $(n,c,k,\mu)$-admissible hypergraph with $\delta_2(\cc G)\geq k$. We  prove that $\cc G$ has a spanning Euler tour.  
Let $G[X,Y]=\mathscr I(\cG)$ with $X=E, Y=V$. By Lemma \ref{LEM:Xdegree}(iii), %with $\alpha=\beta=k$, 
$G$ has a spanning tree $F$ with $\dg_F(x)\leq 2$ for all $x\in X$. Let $A=\{x\in X\ |\ \dg_F(x)= 2\}$. We have
$$|A|+|\overline A|+|Y|=|V(T)|=|E(F)|+1=2|A|+|\overline A|+1,$$
and so $|A|=|Y|-1$. Since vertices of $\overline A$ are leaves of $F$, they are not cut vertices of $F$. Therefore, the graph $F^*:=F-\overline A$  is also a tree with $Y\subseteq V(F^*)$. Let $O\subseteq Y$ be the set of vertices of odd degree in $F^*$. Observe that $|O|$ is even and $O\neq \varnothing$. 
Let us assume that $O=\{y_1,z_1,\dots, y_t, z_t\}$ where $|O|=2t$, and define the new set of vertices $W=\{w_1,\dots, w_t\}$ and the new set of edges $R=\{w_1y_1, w_1z_1,\dots, w_ty_t, w_tz_t\}$. Let $G^*[X^*,Y]=G-A+ R$ where  $X^*:=(X\setminus A) \cup W$. Intuitively speaking, $G^*$ is obtained from $G$ by removing $A$ (together with the edges incident with it), and adding $W$ to $X$ together with the edges in $R$. Observe that $2\leq  \dg_{G^*}(x)\leq k$ for  $x\in X^*$. 

We shall show that $G^*$ has an even $(X^*,2)$-regular subgraph $Q$, 
which will in turn imply that $G$ has a connected spanning $(X,2)$-regular subgraph and so $\cG$  has a  spanning Euler tour. The reasoning is as follows.  
Since $\dg_{G^*} (w)=2$ for each $w\in W$,  we have $\dg_{Q}(w)=2$. Let $Q^*=Q-W$. We have $\dg_{Q^*}(x)=2$ for $x\in \overline A$, $\dg_{Q^*}(y)$ is odd for $y\in O$, and $\dg_{Q^*}(y)$ is even for $y\in \overline{O}$. Recall that $\dg_{F^*}(x)=2$ for $x\in A$, $\dg_{F^*}(y)$ is odd for $y\in O$, $\dg_{F^*}(y)$ is even for $y\in \overline{O}$, and $F^*$ is connected.  Therefore, $F^*\cup Q^*$ corresponds to  a  spanning Euler tour in $\cG$.

\subsection{Properties of $G^*$}
Suppose on the contrary that $G^*$ does not have an even $(X^*,2)$-regular subgraph. By Theorem \ref{parity-factor}, $G^*$ has a barrier. Let $(S,T)$ be a minimum barrier of $G^*$  (so $S,T\subseteq X^*$ and $S\cap T=\varnothing$), and let $O_1,\dots,O_{q}$ be the $T$-odd components of $G^*\backslash (S\cup T)$. Let $u\in T$. Since $c\geq 2$ and all vertices in $W$ have at least two neighbors, we have $\dg_{G*}(u)\geq 2$. 
By Lemma~\ref{minimum-barrier},  there is no edge between $u$ and the $T$-even components of $G^*-(S\cup T)$, and there is at most one edge between $u$ and each $T$-odd component of $G^*-(S\cup T)$. Thus, $q \ge 2$. For $i\in [q]$, let
\begin{align*} 
		&&
		I_i=V(O_i)\cap X,  
		&&
		J_i=V(O_i)\cap Y.
		&&
\end{align*}

For $i\in [q]$, if $|J_i|<c$, then $I_i=\varnothing$. Moreover, if $k> 2c$, then for $i\in [q]$ with  $c\leq |J_i|< k-c$  we have the following.
\begin{align} \label{powersetineq}
    |I_i|&\leq \mu \sum\nolimits_{\ell=c}^{|J_i|}\binom{|J_i|}{\ell}=\mu\left(\sum\nolimits_{\ell=0}^{|J_i|}\binom{|J_i|}{\ell}-\sum\nolimits_{\ell=0}^{c-1}\binom{|J_i|}{\ell}\right) \nonumber\\
    &\leq  \mu\left(2^{|J_i|}-\sum\nolimits_{\ell=0}^{c-1}\binom{c}{\ell} \right)= \mu\left(2^{|J_i|}-\sum\nolimits_{\ell=0}^{c}\binom{c}{\ell}+1 \right) \nonumber\\
    &\leq \mu\left(2^{k-c-1}-2^{c}+1 \right). 
\end{align}

We claim that 
\begin{align} \label{X-connectivityeq}
    \eps_{G^*}\big((S\cup T)\setminus W, O_i\big )\ge z \quad \forall i\in [q].
\end{align} 
where $z:=3$ if  $c=3$ and $k\in \{3,4\}$, and $z:†=k/2$ otherwise.  To prove this claim, let us fix $i\in [q]$.

Since for $1\leq a\leq n/2$, $\min\big\{x(n-x)\ |\ a\leq x\leq  n-a\big\}=a(n-a)$, 
by Lemma \ref{LEM:Xdegree}(ii) and the construction of ${G^*}$, we have the following. 
\begin{align*}
    \eps_{G^*}\big((S\cup T)\setminus W, O_i\big)   &\geq |N_{G}(J_i)\cap N_{G}(\overline {J_i})| -|A|\geq \frac{2|J_i||\overline {J_i}|}{k-1}-|A|\\
    &\geq 1-n+
    \begin{cases}
       \dfrac{2c(n-c)}{k-1} & \mbox { if }k\leq 2c\mbox{ and }  c\leq |J_i|\leq n-c\\ 
        \dfrac{2(k-c)(n-k+c)}{k-1} & \mbox { if } k>2c\mbox{ and }  k-c\leq |J_i|\leq n-k+c
    \end{cases}\\
    &=:\alpha(n,c,k).
\end{align*}   
To show that \eqref{X-connectivityeq} holds under the above cases, observe that $\alpha(n,3,3)=2n-8\geq 3$ for $n\geq 7$, and  $\alpha(n,3,4)=n-5\geq 3$ for $n\geq 10$. For $k\leq 2c$, since $\alpha(n,c,k)\geq k/2$ is equivalent to $n\geq (4c^2+k^2-3k+2)/(4c-2k+2)$, using Lemma \ref{ineqalg1} and  the fact that $\cc G$ is $(n,c,k,\mu)$-admissible, we have $\alpha(n,c,k)\geq k/2$. For $k> 2c$, since $\alpha(n,c,k)\geq k/2$ is equivalent to $n\geq (4 c^2 - 8 c k + 5 k^2 - 3 k + 2)/(2k-4 c  + 2)$, using Lemma \ref{ineqalg2} and  $(n,c,k,\mu)$-admissibility of $\cc G$, we have $\alpha(n,c,k)\geq k/2$.

Now, using Lemma~\ref{LEM:Xdegree}(i) and \eqref{powersetineq}, we have the following.
\begin{align*}% \label{ineqeg}
    \eps_{G^*}\big((S\cup T)\setminus W, O_i\big)   &\geq |N_{G^*}(J_i)\backslash (W\cup I_i)|\geq |N_{G}(J_i)|-| I_i| - |A|\\\nonumber
&\geq \frac{k(n-1)}{k-1}-(n-1)-| I_i|   =\frac{n-1}{k-1}-| I_i|\\
&\geq \frac{n-1}{k-1}-\begin{cases} 
        0 & \mbox { if }  |J_i|<c\\ %k\leq 2c\mbox{ and }  
        \mu\left(2^{k-c-1}-2^{c}+1 \right)  & \mbox { if }  k>2c\mbox{ and }  c\leq |J_i|< k-c
    \end{cases}\\    
    &=:\beta(n,c,k,\mu).  
\end{align*}  
To show that \eqref{X-connectivityeq} holds under the above cases, observe that $\beta(n,3,3,\mu)=(n-1)/2\geq 3$ for $n\geq 7$, and  $\beta(n,3,4,\mu)=(n-1)/3\geq 3$ for $n\geq 10$. For $k\leq 2c$, since $\beta(n,c,k,\mu)\geq k/2$ is equivalent to $n\geq \binom{k}{2}+1$, using $(n,c,k,\mu)$-admissibility of $\cc G$ and Lemma \ref{ineqalg1}, we have $\beta(n,c,k,\mu)\geq k/2$. For $k> 2c$, since $\beta(n,c,k,\mu)\geq k/2$ is equivalent to $n\geq \mu(k-1)(2^{k-c-1}-2^c+1)+\binom{k}{2}+1$, the $(n,c,k,\mu)$-admissibility of $\cc G$ implies that $\beta(n,c,k,\mu)\geq k/2$. 

%Now, let us fix $\ell\in [q]\backslash\{i\}$. 
To settle the remaining cases, observe that if $|J_i|>n-c$, then $|\overline{J_i}|<c$, and if $k>2c$  and $n-k+c<|J_i|\leq n-c$, then  $c\leq |\overline{J_i}|< k-c$.
Again, by applying Lemma~\ref{LEM:Xdegree}(i), we have that each vertex of $J_\ell$ for $\ell\in [q]\backslash\{i\}$ has at least $(n-1)/(k-1)$ neighbors in $((S\cup T)\backslash W) \cup I_\ell$, hence
\begin{align*}
  %  |(S\cup T)\backslash W|&\geq \frac{n-1}{k-1}- \sum\nolimits_{\ell\in [q]\backslash \{ i\}}|I_\ell|\\ 
     |(S\cup T)\backslash W|&\geq \frac{n-1}{k-1}- |I_\ell|\\ 
    &\geq
    \dfrac{n-1}{k-1}-
    \begin{cases} 
        0 & \mbox { if }  |J_i|>n-c,\\ 
         \mu\left(2^{k-c-1}-2^{c}+1 \right)  & \mbox { if }  k>2c\mbox{ and }  n-k+c<|J_i|\leq n-c.
    \end{cases}
\end{align*}
Let $u\in (S\cup T)\backslash W$. Recall that  $u$ corresponds to an edge in $\cc G$, and so it has at least $c$ neighbors in $Y$. If $|J_i|> n-c$, then $u$ must have at least one neighbor in $J_i$, hence 
$$\eps_{G^*}\big((S\cup T)\setminus W, O_i\big) \geq  |(S\cup T)\backslash W|  \geq  \dfrac{n-1}{k-1}\geq  z.$$ 
If $k>2c$ and  $n-k+c<|J_i|\leq n-c$, then  for at most $\mu\left(2^{k-c-1}-2^{c}+1 \right)$ vertices $u\in (S\cup T)\backslash W$, $N_{G^*}(u)\subseteq \overline {J_i}$, and so 
$$\eps_{G^*}\big((S\cup T)\setminus W, O_i\big) \geq  |(S\cup T)\backslash W| -\mu\left(2^{k-c-1}-2^{c}+1 \right) \geq \dfrac{n-1}{k-1}- 2\mu\left(2^{k-c-1}-2^{c}+1 \right)\geq z.$$

\subsection{Discharging}  \label{dischspaneu}
Let us define the auxiliary bipartite graph $H[S\cup T, \{o_1,\dots,o_q\}]$ %whose vertex set is $S\cup T\cup\{o_1,\dots,o_q\}$ 
such that 
\begin{align*}
	\eps_{H}(x,o_i)= \eps_{G^*}(x,O_i), \mbox{ for }x\in S\cup T, i\in [q].
\end{align*}
Clearly, $\eps_{H}((S\cup T)\setminus W,o_i) =\eps_{G^*}((S\cup T)\setminus W,O_i)$. By Lemma \ref{minimum-barrier}  (ii), $\dg_H(x)=\dg_{G*}(x)$ for $x\in T$, but  $\dg_H(x)\leq\dg_{G*}(x)$ for $x\in S$. We shall define $f_0,f_1,f_2:V(H)\rightarrow \mathbb{R}$  such that $$\sum\nolimits_{x\in V(H)}f_0(x)=\sum\nolimits_{x\in V(H)}f_1(x)=\sum\nolimits_{x\in V(H)}f_2(x).$$ 
First, we let 
\begin{align*}
	f_{0}(x)= \begin{cases}
		k & \text{if}\,\, x\in S, \\
		0 & \text{if}\,\, x\in T, \\
		\dfrac{k}{2}\big(\eps_H(T,o_i)-1\big) & \text{if}\,\, x=o_i, i\in [q].
	\end{cases}
\end{align*}
To define $f_1$, we transfer the value of $f_0(x)$ for $x\in S$ to $o_i$s so that 
every vertex $x\in S$ gives $\frac{\eps_H(x,o_i)}{\dg_H(x)}f_0(x)$ to each $o_i$ for $i\in [q]$. Since $\dg_H(x)=\sum\nolimits_{i=1}^q \eps_H(x,o_i)$, we have $f_{1}(x)=0$ for $x\in S$. More precisely, 
\begin{align*}
	f_{1}(x):= \begin{cases}
		0 & \text{if}\,\, x\in S\cup T, \\
		f_0(x)+ \sum\nolimits_{u\in S} \dfrac{\eps_H(u,x)}{\dg_H(u)}f_0(u)  & \text{if}\,\, x=o_i, i\in [q].
	\end{cases}
\end{align*}

In order to define $f_2$, we transfer the value of $f_1(o_i)s$ to $T$ in the following manner. For each $i\in [q]$, first we transfer $\frac{k \eps_H(x,o_i)}{\dg_H(x)}$ of the value of $o_i$ to each $x\in T\setminus W$, and then we distribute the remaining value equally among the vertices in $T\cap W\cap N_H(o_i)$. Since for $x\in T\setminus W$, we have
$$\sum\nolimits_{i\in [q]} \frac{k \eps_H(x,o_i)}{\dg_H(x)}=\frac{k}{\dg_H(x)}\sum\nolimits_{i\in [q]} \eps_H(x,o_i)=k,$$
$f_2$ can be described as follows. 
\begin{align*}
	f_{2}(x)= \begin{cases}
		0 & \text{if}\,\, x\in S\cup \{o_1,\dots,o_q\}, \\
		k   & \text{if}\,\, x\in T\setminus W, \\
		\sum\nolimits_{i\in [q]} \dfrac{f_1(o_i)-\sigma_i}{\eps_H(T\cap W,o_i)}\eps_H(x,o_i)        & \text{if}\,\,  x\in T\cap W,
	\end{cases}
\end{align*}
where 
$$
\sigma_i:=\sum\nolimits_{x\in T\setminus W} \frac{k \eps_H(x,o_i)}{\dg_H(x)}  \mbox { for } i\in [q].
$$
Recall that for $x\in T\setminus W$, $\dg_H(x)=\dg_{G^*}(x)\geq c$, so we have
\begin{align} \label{sigmiineq}
    \sigma_i\leq\sum\nolimits_{x\in T\setminus W} \frac{k \eps(x,o_i)}{c} =\frac{k}{c}\eps_H(T\backslash W,o_i) \mbox { for } i\in [q].
\end{align}

We claim that in order to complete the proof, it suffices to show that
\begin{align} \label{claimf2}
	\frac{f_1(o_i)-\sigma_i}{\eps_H(T\cap W,o_i)}\geq \frac{k}{2} \mbox{ for } i\in [q]. 
\end{align}
To see this, note that if \eqref{claimf2} holds, then  for $x\in T\cap W$,
\begin{align*}
	f_2(x)\geq \sum\nolimits_{i\in [q]} \frac{k}{2}\eps_H(x,o_i)=    \frac{k}{2}\sum\nolimits_{i\in [q]} \eps_H(x,o_i)= \frac{k}{2}\dg_H(x)= k.
\end{align*}
Since $f_2(x)\geq k$ for each $x\in T$, we have 
\begin{align*}
	k|T|&\leq  \sum\nolimits_{x\in T}f_2(x)= \sum\nolimits_{x\in S}f_0(x)+\sum\nolimits_{i\in [q]} f_0(x)\\
	&= k|S|+ \frac{k}{2}\sum\nolimits_{i\in [q]}\big(\eps_H(T,o_i)-1\big).
\end{align*}
Thus,  
$$|T|-|S|\leq \frac{1}{2} \sum\nolimits_{i\in [q]}\big(\eps(T,C_i)-1\big),$$
contradicting Lemma~\ref{minimum-barrier}(iv). Therefore, $G^*$ has an even $(X^*,2)$-regular subgraph, completing the proof.

Now, let us fix $i\in [q]$, and let $s=\eps_{H}(S\setminus W, o_i), t=\eps_H(T\backslash W,o_i)$. Observe that $\eps_H(T,o_i)=t+\eps_H(T\cap W,o_i)$. Recall that $\dg_{G^*}(u)\leq k$ for $k\in S$, and  by \eqref{X-connectivityeq}, $2s\geq 6-2t$ if $c=3, k\in \{3,4\}$, and $2s\geq k-2t$ otherwise. 
We apply  \eqref{sigmiineq}  to obtain the following which proves \eqref{claimf2}.

\begin{align*}
	2f_1(o_i)-2\sigma_i-k\eps_H(T\cap W,o_i) &\geq 2\sum\nolimits_{u\in S} \frac{k \eps_H(u,o_i)}{\dg_H(u)} + 2f_0(o_i)-\frac{2kt}{c}- k\eps_H(T\cap W,o_i)\\ 
 &\geq 2\sum\nolimits_{u\in S} \frac{k \eps_H(u,o_i)}{\dg_{G^*}(u)}+ k\big(\eps_H(T,o_i)-1\big) -\frac{2kt}{c}- k\eps_H(T\cap W,o_i)\\ 
 &\geq 2\sum\nolimits_{u\in S} \eps_H(u,o_i)+kt-k-\frac{2kt}{c}\\ 
 &= 2\eps_H(S,o_i)+kt-k-\frac{2kt}{c}\\ 
 &\geq  k t-k + 2s-\frac{2kt}{c}\\
&\geq 
    \begin{cases}
        k t-k +(k-2t)-\dfrac{2kt}{4} &\mbox { if  }c\geq 4\\
         2s-k+\dfrac{kt}{3}  &\mbox { if  }c=3
    \end{cases}\\
    &\geq  
    \begin{cases}
        k t-2 t -\dfrac{kt}{2} &\mbox { if  }c\geq 4\\
         (k-2t)-k+\dfrac{kt}{3}  &\mbox { if  }c=3, k\geq 6\\ 
         (6-2t)-k+\dfrac{kt}{3}   &\mbox { if  }c=3, k\in \{3,4,5\}, t\leq 3\\
         2s   &\mbox { if  }c=3, k\in \{3,4,5\},t>3
    \end{cases}\\
&\geq
\begin{cases}
        (k-4)t/2 &\mbox { if  }c\geq 4,\\ 
        (k-6)t/3  &\mbox { if  }c=3, k\geq 6,\\ 
         (3-t)(6-k)/3   &\mbox { if  }c=3, k\in \{3,4,5\}, t\leq 3,\\
         2s   &\mbox { if  }c=3, k\in \{3,4,5\},t>3.
    \end{cases}
\end{align*}
\qed 

\subsection{Proof of Theorem \ref{mainlong}} \label{spaneufinal}
Let $\cc G$ be an $(n,c,k,\mu)$-admissible hypergraph, and let $2< r\leq k$. 
It suffices to show that 
\begin{align} \label{mainnbounddelt}
\delta_r(\cc G)> (k-1)\dbinom{k-2}{r-2}\Big/\dbinom{n-2}{r-2},
\end{align}
for by Theorem \ref{enumlem}, 
$$
\delta_2(\cc G)\geq \delta_r(\cc G)\dbinom{n-2}{r-2}\Big/\dbinom{k-2}{r-2}>k-1,
$$
and so by the results of subsections \ref{consspaneu}--\ref{dischspaneu}, $\cc G$ has a spanning Euler tour. To prove \eqref{mainnbounddelt}, first recall that $\cc G$ is an $(n,c,k,\mu)$-admissible hypergraph, and so if $4\leq r\leq k$, then $n\geq \binom{k}{2}+1$. 
Using Bernoulli's inequality ($\forall p\geq 1, \forall x\geq -1: (1+x)^p\geq 1+px$), we have the following which completes the proof.%recall that when $r=3$, we have $\delta_3(\cc G)\geq 1>(k-1)(k-2)/(n-2)$ for $n> k^2-3k+4$. Finally, if $r=4$, $\delta_r(\cc G)\geq 1>(k-1)\dbinom{k-2}{r-2}\Big/\dbinom{n-2}{r-2}$. 
\begin{align*}
    \dbinom{n-2}{r-2}\bigg/\dbinom{k-2}{r-2}&= \prod_{i=2}^{r-1}\frac{n-i}{k-i}=
\prod_{i=2}^{r-1}\left(1 + \frac{n-k}{k-i}\right)\geq \left(1 + \frac{n-k}{k-2}\right)^{r-2}\geq 1+\frac{(n-k)(r-2)}{k-2}\\
   &\geq \begin{cases}
        1+\dfrac{n-k}{k-2}=\dfrac{n-2}{k-2}>k-1 &\mbox { if  }r=3, n> k^2-3k+4,\\
        1+\dfrac{2(n-k)}{k-2}=\dfrac{2n-k-2}{k-2}>k-1  &\mbox { if  }r\geq 4.
    \end{cases}\\
\end{align*}

\bibliographystyle{plain}   
  % \bibliography{eulerhypbib}

\end{document}